\documentclass{amsart}
\usepackage{amsfonts, amsbsy, amsmath, amssymb}

\newtheorem{thm}{Theorem}[section]
\newtheorem{lem}[thm]{Lemma}

\newtheorem{prop}[thm]{Proposition}

\newtheorem{thm-con}[thm]{Theorem-Conjecture}
\numberwithin{equation}{section}

\theoremstyle{definition}

\allowdisplaybreaks

\newcommand{\f}{\Bbb F}

\begin{document}

\title{On a Class of Permutation Trinomials in Characteristic 2}

\author[Xiang-dong Hou]{Xiang-dong Hou}
\address{Department of Mathematics and Statistics,
University of South Florida, Tampa, FL 33620}
\email{xhou@usf.edu}



\begin{abstract}
Recently, Tu, Zeng, Li, and Helleseth considered trinomials of the form $f(X)=X+aX^{q(q-1)+1}+bX^{2(q-1)+1}\in\f_{q^2}[X]$, where $q$ is even and $a,b\in\f_{q^2}^*$. They found sufficient conditions on $a,b$ for $f$ to be a permutation polynomial (PP) of $\f_{q^2}$ and they conjectured that the sufficient conditions are also necessary. The conjecture has been confirmed by Bartoli using the Hasse-Weil bound. In this paper, we give an alternative solution to the question. We also use the Hasse-Weil bound, but in a different way. Moreover, the necessity and sufficiency of the conditions are proved by the same approach.
\end{abstract}

\maketitle

\section{Introduction}

Let $\f_q$ denote the finite field with $q$ elements. A polynomial $f\in\f_q[X]$ is called a {\em permutation polynomial} (PP) of $\f_q$ if it induces a permutation of $\f_q$. PPs of the form
\begin{equation}\label{1.1}
f(X)=X+aX^{s_1(q-1)+1}+bX^{s_2(q-1)+1}\in\f_{q^2}[X],\quad 1\le s_1,s_2\le q,\ s_1\ne s_2,
\end{equation}
have attracted much attention in recent years \cite{Gupta-Sharma-FFA-2016, Hou-CM-2015, Hou-FFA-2015a, Hou-FFA-2015b, Li-Helleseth-CC-2017, Li-Qu-Chen-FFA-2017, Wu-Yuan-Ding-Ma-FFA-2017, Zha-Hu-Fan-FFA-2017, Zieve-arXiv1310.0776}. Given $(s_1,s_2)$, finding conditions on $a,b$ that are necessary and sufficient for $f$ to be a PP of $\f_{q^2}$ is a difficult question that offers not only challenges but also fascination. The ``simplest'' cases with $(s_1,s_2)=(1,2)$ was solved a few years ago \cite{Hou-FFA-2015b}. In a recent paper \cite{Tu-Zeng-Li-Helleseth-FFA-2018}, Tu, Zeng, Li, and Helleseth considered the case $(s_1,s_2)=(q,2)$ with even $q$. Let 
\begin{equation}\label{1.2}
f(X)=X+aX^{q(q-1)+1}+bX^{2(q-1)+1}\in\f_{q^2}[X],
\end{equation}
where $q$ is even and $a,b\in\f_{q^2}^*$. They proved that $f$ is a PP of $\f_{q^2}$ if
\begin{equation}\label{1.3}
b(1+a^{q+1}+b^{q+1})+a^{2q}=0
\end{equation}
and
\begin{equation}\label{1.4}
\begin{cases}
\displaystyle\text{Tr}_{q/2}\Bigl(1+\frac 1{a^{q+1}}\Bigr)=0&\text{if}\ b^{q+1}=1,\vspace{0.4em}\cr
\displaystyle\text{Tr}_{q/2}\Bigl(\frac {b^{q+1}}{a^{q+1}}\Bigr)=0&\text{if}\ b^{q+1}\ne 1,
\end{cases}
\end{equation}
where $\text{Tr}_{q/2}$ is the trace from $\f_q$ to $\f_2$. Based on numerical experiments, the authors conjectured that the conditions \eqref{1.3} and \eqref{1.4} are also necessary for $f$ to be a PP of $\f_{q^2}$. The conjecture has been proved by Bartoli \cite{Bartoli-arXiv:1712.10017}. If $f$ is a PP of $\f_{q^2}$, it is well known that there is an associated rational function $F(X)\in\f_q(X)$ of degree $3$ which permutes $\f_q$. The Hasse-Weil bound implies that when $q$ is not too small, the numerator of $(F(X)-F(Y))/(X-Y)$ does not have absolutely irreducible factors in $\f_q[X]$. With computer assistance, \cite{Bartoli-arXiv:1712.10017} determined the necessary and sufficient conditions for the numerator of $(F(X)-F(Y))/(X-Y)$ not to have absolutely irreducible factors in $\f_q[X]$, and these conditions are equivalent to \eqref{1.3} and \eqref{1.4}.

In the present paper, we give a different proof for the results of  \cite{Bartoli-arXiv:1712.10017} and \cite{Tu-Zeng-Li-Helleseth-FFA-2018}. We also use the Hasse-Weil bound, but in a different way. Moreover, we prove the necessity and sufficiency of the conditions \eqref{1.3} and \eqref{1.4} at the same time. The paper is organized as follows. Section 2 contains some preliminary steps of the proof. We observe that after a simple substitution, we may assume that $b\in\f_q^*$. We also recall a few known results to be used later. In Section 3, we use the Hasse-Weil bound to show that when $q$ is not too small, $f$ is a PP of $\f_{q^2}$ essentially if and only if a certain polynomial in $\f_q[X,Y]$ factors in a prescribed manner; the factorization is impossible unless $a\in\f_q^*$. In this section, the reader will find that heavy computations can produce surprisingly simple results, a phenomenon, though mysterious, not uncommon in the study of PPs. Section 4 is a ``rerun'' of the computations in Section 3 under the assumption that $a\in\f_q^*$; the computational results confirm that the conditions \eqref{1.3} and \eqref{1.4} are necessary and sufficient for $f$ to be a PP of $\f_{q^2}$. Since we assume that $b\in\f_q^*$, the conditions \eqref{1.3} and \eqref{1.4} become simpler. The main result of the paper is the following 

\begin{thm}\label{T1.1}
Let $q$ be even and $f(X)=X+aX^{q(q-1)+1}+bX^{2(q-1)+1}$, where $a\in\f_{q^2}^*$ and $b\in\f_q^*$. Then $f$ is a PP of $\f_{q^2}$ if and only if 
\begin{itemize}
\item[(i)] $b=1$, $a\in\f_q^*$ and $\text{\rm Tr}_{q/2}(1+a^{-1})=0$, or
\item[(ii)] $b\ne 1$, $\text{\rm Tr}_{q/2}(b/(b+1))=0$ and $a^2=b(b+1)$.
\end{itemize}
\end{thm}

We leave it for the reader to verify that under the assumption $b\in\f_q^*$, the intersection of \eqref{1.3} and \eqref{1.4} is equivalent to the union of (i) and (ii).

The computations in the paper require no specialized algorithms and the results can be easily verified with computer assistance. The proof produces and uses some lengthy expressions at various stages; these expressions are given in the appendix.

\section{Preliminaries}

From now on, $q$ is even and 
\begin{equation}\label{2.1}
f=X(1+aX^{q(q-1)}+bX^{2(q-1)})\in\f_{q^2}[X],
\end{equation}
where $a,b\in\f_{q^2}^*$. Let $\beta\in\f_{q^2}$ be such that $\beta^4=b$. Then
\[
f(\beta X)=\beta X(1+a\beta^{1-q}X^{q(1-q)}+\beta^{2(q+1)}X^{2(q-1)}),
\]
where $\beta^{2(q+1)}\in\f_q^*$. Since $f(X)$ is a PP of $\f_{q^2}$ if and only if $f(\beta X)$ is, we may assume that $b\in\f_q^*$ in $f(X)$.

Let $\mu_{q+1}=\{x\in\f_{q^2}^*:x^{q+1}=1\}$. It is well known that $f$ is a PP of $\f_{q^2}$ if and only if $h(X)=X(1+aX^q+bX^2)^{q-1}$ permutes $\mu_{q+1}$ \cite{Park-Lee-BAMS-2001, Wang-LNCS-2007, Zieve-PAMS-2009}. For $x\in\mu_{q+1}$ with $1+ax^q+bx^2\ne 0$, i.e., $bx^3+x+a\ne 0$, we have
\[
h(x)=\frac{x(1+ax^q+bx^2)^q}{1+ax^q+bx^2}=\frac{a^qx^3+x^2+b}{bx^3+x+a}.
\]
Hence $h(X)$ permutes $\mu_{q+1}$ if and only if $bx^3+x+a\ne 0$ for all $x\in\mu_{q+1}$ and 
\[
g(X)=\frac{a^qX^3+X^2+b}{bX^3+X+a}\in\f_{q^2}(X)
\]
permutes $\mu_{q+1}$.

Assume that $bX^3+X+a$ has no root in $\mu_{q+1}$. Choose $k\in\f_q$ such that $\text{Tr}_{q/2}(k)=1$ and let $z\in\f_{q^2}$ be such that 
\begin{equation}\label{2.2}
z^2+z+k=0.
\end{equation}
Then $z+z^q=1$ and $z^{q+1}=k$. The rational function $\phi(X)=(X+z^q)/(X+z)$ maps $\f_q\cup\{\infty\}$ to $\mu_{q+1}$ bijectively with $\phi(\infty)=1$ and $g(1)=(1+a+b)^{q-1}$. Hence $g(X)$ permutes $\mu_{q+1}$ if and only if $\psi^{-1}\circ g\circ\phi$ permutes $\f_q$, where $\psi(X)=(1+a+b)^{q-1}\phi(X)$, i.e., if and only if for each $y\in\f_q$, there is a unique $x\in\f_q$ such that $g\circ\phi(x)=(1+a+b)^{q-1}\phi(y)$. To summarize, we have the following proposition.

\begin{prop}\label{P2.1}
$f$ is a PP of $\f_{q^2}$ if and only if 
\begin{itemize}
\item[(i)] $bX^3+X+a$ has no root in $\mu_{q+1}$, and
\item[(ii)] for each $y\in\f_q$, there is a unique $x\in\f_q$ such that
\begin{equation}\label{2.3}
g\Bigl(\frac{x+z^q}{x+z}\Bigr)=(1+a+b)^{q-1}\frac{y+z^q}{y+z}.
\end{equation}
\end{itemize}
\end{prop}

We will also need the following result.

\begin{lem}[\cite{Williams-JNT-1975}]\label{L2.2} 
Let $\alpha,\beta\in\f_{2^n}$, $\beta\ne 0$. The polynomial $X^3+\alpha X+\beta$ has exactly one root in $\f_{2^n}$ if and only if $\text{\rm Tr}_{2^n/2}(1+\alpha^3\beta^{-2})=1$.
\end{lem}

\section{Necessity that $a\in\f_q^*$}

In this section we prove the following claim.

\begin{prop}\label{P3.1}
If $f$ is a PP of $\f_{q^2}$ and $q\ge 2^6$, then $a\in\f_q^*$.
\end{prop}

\begin{proof} 
Assume to the contrary that $a\in\f_{q^2}\setminus\f_q$. We will eventually arrive at a contradiction.

\medskip
$1^\circ$ Let $z=a/(a+a^q)\in\f_{q^2}\setminus\f_q$. Then $z^2+z=k$, where $k=a^{q+1}/(a+a^q)^2\in\f_q$. Write $a=a_1z$, where $a_1=a+a^q\in\f_q^*$. We have
\begin{equation}\label{3.1}
g\Bigl(\frac{X+z^q}{X+z}\Bigr)=g\Bigl(\frac{X+z+1}{X+z}\Bigr)=\frac{A(X)}{B(X)},
\end{equation}
where
\begin{align}\label{3.2}
A(X)=\,&X^3 (1+a_1+b+a_1 z)+X^2 (a_1+a_1 k+z+a_1z+b z)\cr
&+X (1+a_1+k+b k+z+a_1 z+b z+a_1 k z)\cr
&+a_1+k+a_1 k+b k+a_1 k^2+a_1 z+b z+k z+b k z
\end{align}
and
\begin{align}\label{3.3}
B(X)=\,&X^3 (1+b+a_1 z)+X^2 (1+b+a_1 k+z+a_1 z+b z)\cr
&+X (b+k+a_1 k+b k+z+a_1 z+b z+a_1 k z)\cr
&+b+a_1 k+a_1 k^2+a_1 z+b z+k z+b k z.
\end{align}
By Proposition~\ref{P2.1}, $B(X)$ has no root in $\f_q$, and for each $y\in\f_q$, there is a unique $x\in\f_q$ such that
\begin{equation}\label{3.4}
\frac{A(x)}{B(x)}=\frac {1+a_1+b+a_1z}{1+b+a_1z}\cdot\frac{y+z+1}{y+z}.
\end{equation}
Write \eqref{3.4} as
\begin{equation}\label{3.5}
C_3x^3+C_2(y)x^2+C_1(y)x+C_0(y)=0,
\end{equation}
where
\begin{gather}
C_3=1+a_1+a_1 b+b^2+a_1^2 k, \label{3.6} \\
C_2(Y)=1+a_1+a_1 b+b^2+(1+b^2+a_1^2 k )Y, \label{3.7} \\
C_1(Y)=b+a_1 b+b^2+k+a_1 k+a_1 b k+b^2 k+a_1^2 k^2+(1+a_1+b^2+a_1^2 k)Y, \label{3.8} \\
C_0(Y)=b+a_1 b+b^2+a_1 k+a_1^2 k^2+(a_1+b+b^2+k +a_1^2 k +b^2 k +a_1^2 k^2 )Y. \label{3.9} 
\end{gather}
We claim that $C_3\ne 0$. Otherwise,
\[
k=\Bigl(\frac{1+b}{a_1}\Bigr)^2+\frac{1+b}{a_1},
\]
which is impossible since $\text{Tr}_{q/2}(k)=1$. We write $C_i=C_i(Y)\in\f_q[Y]$ and $c_i=C_i(y)$. Now \eqref{3.5} becomes
\[
x^3+\frac{c_2}{c_3}x^2+\frac{c_1}{c_3}x+\frac{c_0}{c_3}=0,
\]
i.e.,
\begin{equation}\label{3.10}
{x'}^3+\frac{c_2^2+c_1c_3}{c_3^2}x'+\frac{c_1c_2+c_0c_3}{c_3^2}=0,
\end{equation}
where $x'=x+c_2/c_3$. By \eqref{3.10} and Lemma~\ref{L2.2}, for each $y\in\f_q$, we have 
\begin{itemize}
\item[(i)] $c_1c_2+c_0c_3=c_2^2+c_1c_3=0$, or
\item[(ii)] $c_1c_2+c_0c_3\ne 0$ and
\begin{equation}\label{3.11}
\text{Tr}_{q/2}\Bigl(1+\frac{(c_2^2+c_1c_3)^3}{c_3^2(c_1c_2+c_0c_3)^2}\Bigr)=1.
\end{equation}
\end{itemize}

\medskip
$2^\circ$ For each $y\in\f_q$ with $(C_1C_2+C_0C_3)(y)\ne 0$, by \eqref{3.11} we have
\begin{align*}
1\,&=\text{Tr}_{q/2}\Bigl(1+\frac{(c_2^2+c_1c_3)^3}{c_3^2(c_1c_2+c_0c_3)^2}\Bigr)\cr
&=\text{Tr}_{q/2}\Bigl(1+\frac{c_2^6+c_2^2c_1^2c_3^2+c_2^4c_1c_3+c_1^3c_3^3}{c_3^2(c_1c_2+c_0c_3)^2}\Bigr)\cr
&=\text{Tr}_{q/2}\Bigl(1+\frac{c_2^3+c_1c_2c_3}{c_3(c_1c_2+c_0c_3)}+\frac{c_1c_2^4+c_1^3c_3^2}{c_3(c_1c_2+c_0c_3)^2}\Bigr)\cr
&=\text{Tr}_{q/2}\Bigl(1+\frac{(c_2^3+c_1c_2c_3)(c_1c_2+c_0c_3)+c_1c_2^4+c_1^3c_3^2}{c_3(c_1c_2+c_0c_3)^2}\Bigr)\cr
&=\text{Tr}_{q/2}\Bigl(1+\frac{(c_1^2+c_0c_2)(c_2^2+c_1c_3)}{(c_1c_2+c_0c_3)^2}\Bigr).
\end{align*}
Hence there a precisely two $x\in\f_q$ such that 
\begin{equation}\label{3.12}
x^2+x=k+1+\frac{(c_1^2+c_0c_2)(c_2^2+c_1c_3)}{(c_1c_2+c_0c_3)^2}.
\end{equation}
Write
\[
\frac{(C_1^2+C_0C_2)(C_2^2+C_1C_3)}{(C_1C_2+C_0C_3)^2}=\frac PQ,
\]
where $P,Q\in\f_q[Y]$ and $\text{gcd}(P,Q)=1$. Let
\begin{equation}\label{3.13}
F(X,Y)=Q(Y)(X^2+X+k+1)+P(Y)\in\f_q[X,Y]
\end{equation}
and
\begin{equation}\label{3.14}
V_{\f_q^2}(F)=\{(x,y)\in\f_q^2:F(x,y)=0\}.
\end{equation}
It is clear that all points on the curve $V_{\f_q^2}(F)$ are smooth, and by \eqref{3.12},
\begin{equation}\label{3.15}
|V_{\f_q^2}(F)|\ge 2(q-2).
\end{equation}

\medskip
$3^\circ$ We claim that $F(X,Y)$ is not irreducible over $\overline\f_q$. Otherwise, let ${\tt x}, {\tt y}$ be transcendentals over $\f_q$ satisfying $F({\tt x},{\tt y})=0$. By Riemann's inequality \cite[III.10.4]{Stichtenoth-1993}, the functional field $\f_q({\tt x},{\tt y})/\f_q$ has genus 
\[
g\le([\f_q({\tt x},{\tt y}):\f_q({\tt x})]-1)([\f_q({\tt x},{\tt y}):\f_q({\tt y})]-1)\le(4-1)(2-1)=3. 
\]
Then by the Hasse-Weil bound \cite[V.2.3]{Stichtenoth-1993},
\[
|V_{\f_q^2}(F)|\le q+1+2gq^{1/2}\le q+1+6q^{1/2}<2(q-2),
\]
which is a contradiction to \eqref{3.15}.

Now we can write $F=eG_1G_2$, where $e\in\f_q^*$, $G_1,G_2\in\overline\f_q[X,Y]$ are irreducible and monic in some term order and $\deg_XG_1=\deg_XG_2=1$. If $G_1\notin\f_q[X,Y]$, choose $\sigma\in\text{Aut}(\overline\f_q/\f_q)$ such that $\sigma G_1\ne G_1$. Then $\sigma G_1=G_2$ and hence
\[
V_{\f_q^2}(F)\subset V_{\f_q^2}(G_1)\cap V_{\f_q^2}(\sigma G_1).
\]
By B\'ezout's theorem,
\[
|V_{\f_q^2}(F)|\le |V_{\f_q^2}(G_1)\cap V_{\f_q^2}(\sigma G_1)|\le(\deg G_1)^2\le 4,
\]
which is a contradiction to \eqref{3.15}. Therefore $G_1,G_2\in\f_q[X,Y]$. Thus
\[
\frac{F(X,Y)}{Q(Y)}=X^2+X+k+1+\frac{(C_1^2+C_0C_2)(C_2^2+C_1C_3)}{(C_1C_2+C_0C_3)^2}
\]
is a product of two linear polynomials in $X$ over $\f_q(Y)$, namely,
\begin{align}\label{3.16}
& X^2+X+k+1+\frac{(C_1^2+C_0C_2)(C_2^2+C_1C_3)}{(C_1C_2+C_0C_3)^2}\\
=\,&\Bigl(X+\frac D{C_1C_2+C_0C_3}\Bigr)\Bigl(X+1+\frac D{C_1C_2+C_0C_3}\Bigr) \nonumber
\end{align}
for some $D\in\f_q[Y]$.

\medskip
$4^\circ$ Equation~\eqref{3.16} is equivalent to 
\begin{equation}\label{3.17}
D(D+C_1C_2+C_0C_3)=(k+1)(C_1C_2+C_0C_3)^2+(C_1^2+C_0C_2)(C_2^2+C_1C_3).
\end{equation}
Write 
\begin{gather}
D=D_2Y^2+D_1Y+D_0, \label{3.18}\\
C_1C_2+C_0C_3=E_2Y^2+E_1Y+E_0, \label{3.19}\\
(k+1)(C_1C_2+C_0C_3)^2+(C_1^2+C_0C_2)(C_2^2+C_1C_3)=F_4Y^4+\cdots+F_0,\label{3.20}
\end{gather}
where 
\begin{align}
E_2\,&=1+a_1+a_1 b^2+b^4+a_1^3 k+a_1^4 k^2,\label{3.21} \\
E_1\,&=1+a_1+a_1 b+a_1 b^2+a_1 b^3+b^4+a_1^3 k+a_1^3 b k+a_1^4 k^2, \label{3.22} \\
E_0\,&=k+a_1 k+a_1^3 b k+a_1 b^2 k+b^4 k+a_1^3 k^2+a_1^4 k^3, \label{3.23} 
\end{align}
and $F_4,\dots,F_0$ are given in Appendix \eqref{s3-F4} -- \eqref{s3-F0}. Comparing the coefficients of $Y^i$, $0\le i\le 4$, in \eqref{3.17} gives
\begin{align}
&D_2^2+E_2D_2=F_4,\label{3.29}\\
&E_1D_2+E_2D_1=F_2,\label{3.30}\\
&E_0D_2+D_1^2+E_1D_1+E_2D_0=F_2,\label{3.31}\\
&E_0D_1+E_1D_0=F_1,\label{3.32}\\
&D_0^2+E_0D_0=F_0.\label{3.33}
\end{align}
Using \eqref{3.30} and \eqref{3.32} in \eqref{3.29}, \eqref{3.31} and \eqref{3.33}, we obtain the following equations in $D_1$:
\begin{align}
&E_2^2D_1^2+E_1E_2^2D_1+F_3^2+E_1E_2F_3+E_1^2F_4=0,\label{3.34}\\
&E_1D_1^2+E_1^2D_1^2+E_0F_3+E_2F_1+E_1F_2=0,\label{3.35}\\
&E_0^2D_1^2+E_0^2E_1D_1+F_1^2+E_0E_1F_1+E_1^2F_0=0.\label{3.36}
\end{align}
Eliminating $D_1$ in the above gives
\begin{align}
&E_1(F_3^2+E_1E_2F_3+E_1^2F_4)+E_2^2(E_0F_3+E_2F_1+E_1F_2)=0,\label{3.37}\\
&E_1(F_1^2+E_0E_1F_1+E_1^2F_0)+E_0^2(E_0F_3+E_2F_1+E_1F_2)=0.\label{3.38}
\end{align}
By \eqref{3.21} -- \eqref{3.23} and \eqref{s3-F4} -- \eqref{s3-F0},
\begin{align}
\text{the left side of \eqref{3.37}}\,&=a_1^2(1+a_1+a_1b+b^2+a_1^2k)^3h_1,\label{3.39} \\
\text{the left side of \eqref{3.38}}\,&=(1+a_1+a_1b+b^2+a_1^2k)^3h_2,\label{3.40}
\end{align}
where $h_1$ and $h_2$ are given in Appendix \eqref{s3-h1} and \eqref{s3-h2}.
Note that $1+a_1+a_1b+b^2+a_1^2k\ne 0$ since $\text{Tr}_{q/2}(k)=1$. Hence $h_1=0$ and $h_2=0$. Using suitable combinations of $h_1$ and $h_2$ to reduce the degree in $k$, we arrive at the following equations:
\begin{align}\label{3.43}
&a_1^4(1+a_1^2+b+b^2+b^3)^2k+a_1^2(1+a_1^4+a_1^4b^2+a_1^2b^3+a_1^2b^7+b^8)\\
=\,&(a_1^2k^2+a_1^2bk+b+b^3)h_1+h_2=0, \nonumber
\end{align}
\begin{equation}\label{3.44}
a_1^8b^3(1+b)^4(1+a_1^2+a_1b+a_1b^2+b^4)^4=d_1h_1+d_2h_2=0,
\end{equation}
where $d_1,d_2$ are given in Appendix \eqref{s3-d1} and \eqref{s3-d2}. We claim that $b\ne 1$. Otherwise, by \eqref{3.43}, $k=0$, which is a contradiction. Thus by \eqref{3.44},
\begin{equation}\label{3.47}
1+a_1^2+a_1b+a_1b^2+b^4=0.
\end{equation}
If $1+a_1^2+b+b^2+b^3=0$, then this equation and \eqref{3.47} together imply $b(1+b)=0$, which is a contradiction. Hence $1+a_1^2+b+b^2+b^3\ne 0$. Now by \eqref{3.43},
\begin{align*}
\text{Tr}_{q/2}(k)=\,&\text{Tr}_{q/2}\Bigl(\frac{1+a_1^4+a_1^4b^2+a_1^2b^3+a_1^2b^7+b^8}{a_1^2(1+a_1^2+b+b^2+b^3)^2}\Bigr)\cr
=\,&\text{Tr}_{q/2}\Bigl(\frac{(1+b)^8+a_1^4(1+b)^2+a_1^2b^3(1+b)^4}{a_1^2(a_1^2+(1+b)^3)^2}\Bigr)\cr
=\,&\text{Tr}_{q/2}\Bigl(\frac{(1+b)^2(a_1^4+(1+b)^6)}{a_1^2(a_1^2+(1+b)^3)^2}+\frac{b^3(1+b)^4}{(a_1^2+(1+b)^3)^2}\Bigr)\cr
=\,&\text{Tr}_{q/2}\Bigl(\frac{1+b}{a_1}+\frac{b^3(1+b)^4}{b^2(1+b)^2(1+b^2+a_1)^2}\Bigr)\cr
&(\text{by \eqref{3.47},}\ a_1^2+(1+b)^3=b(1+b)(1+b^2+a_1))\cr
=\,&\text{Tr}_{q/2}\Bigl(\frac{1+b}{a_1}+\frac{b(1+b)^2}{(1+b^2+a_1)^2}\Bigr)\cr
=\,&\text{Tr}_{q/2}\Bigl(\frac{1+b}{a_1}+\frac{1+b}{a_1}\Bigr)\cr
&(\text{by \eqref{3.47},}\ (1+b^2+a_1)^2=a_1b(1+b))\cr
=\,&0,
\end{align*}
which is a contradiction. This completes the proof of Proposition~\ref{P3.1}.
\end{proof}

\section{Proof of Theorem~\ref{T1.1}}

We now prove that the conditions (i) and (ii) in Theorem~\ref{T1.1} are necessary and sufficient for $f$ to be a PP of $\f_{q^2}$.

\subsection{Necessity}\

Since Theorem~\ref{T1.1} has been verified numerically for $q\le 2^7$ \cite{Tu-Zeng-Li-Helleseth-FFA-2018}, we assume that $q\ge 2^6$. By Proposition~\ref{P3.1}, $a\in\f_q^*$. Choose $k\in\f_q$ such that $\text{Tr}_{q/2}(k)=1$ and let $z\in\f_{q^2}$ be such that $z^2+z=k$. We will go through the computations in Section~3 again. However, since $a\in\f_q^*$, the computations are simpler.

For \eqref{3.2} and \eqref{3.3}, we have
\begin{align}\label{4.1}
A(X)=\,& (1+a+b) X^3+(a+z+a z+b z)X^2 +(1+a+k+a k+b k+z+a z+b z)X\\
& +a+k+b k+a z+b z+k z+a k z+b k z, \nonumber
\end{align}
\begin{align}\label{4.2}
B(X)=\,& (1+a+b) X^3+ (1+b+z+a z+b z)X^2+ (b+k+a k+b k+z+a z+b z)X\\
&+b+a k+a z+b z+k z+a k z+b k z. \nonumber
\end{align}
For \eqref{3.6} -- \eqref{3.9},
\begin{align}
C_3\,&=1+a+b,\label{4.3}\\
C_2(Y)\,&=1+b+(1+a+b)Y,\label{4.4}\\
C_1(Y)\,&=b+(1+a+b)k+(1+a+b)Y,\label{4.5}\\
C_0(Y)\,&=b+ak+(a+b+k+ak+bk)Y.\label{4.6}
\end{align}
Note that $C_3\ne 0$ since $bX^3+X+a$ has no root in $\mu_{q+1}$. For \eqref{3.21} -- \eqref{3.23},
\begin{align}
E_2\,&=1+a^2+b^2,\label{4.7}\\
E_1\,&=1+a^2+b^2,\label{4.8}\\
E_0\,&=ab+k+a^2k+b^2k.\label{4.9}
\end{align}
For \eqref{s3-F4} -- \eqref{s3-F0},
\begin{align}
F_4=\,& a+a^2+a^3+a^4+b+a^2 b+b^2+a b^2+b^3+b^4,\label{4.10}\\
F_3=\,& 1+a^2+b+a^2 b+b^2+b^3,\label{4.11}\\
F_2=\,& a^3+a^4+b+b^2+b^3+b^4,\label{4.12}\\
F_1=\,& a+b+a b+b^3+k+a^2 k+b k+a^2 b k+b^2 k+b^3 k,\label{4.13}\\
F_0=\,& b+b^2+a b^2+a^2 b^2+a k+b k+a b^2 k+b^3 k+a k^2+a^2 k^2 \label{4.14}\\
&+a^3 k^2+a^4 k^2+b k^2+a^2 b k^2+b^2 k^2+a b^2 k^2+b^3 k^2+b^4 k^2. \nonumber
\end{align}
Equation \eqref{3.37} becomes ``$0=0$'', but \eqref{3.38} becomes
\begin{equation}\label{4.15}
(1+b)^3(1+a+b)(a^2+b+b^2)=0.
\end{equation}
Therefore $b=1$ or $a^2=b(b+1)$. By \eqref{4.7} and \eqref{4.8}, $E_1E_2\ne 0$, and by \eqref{3.34}, 
\begin{align}\label{4.16}
0\,&=\text{Tr}_{q/2}\Bigl(\frac{E_2^2D_1^2+E_1E_2^2D_1+F_3^2+E_1E_2F_3+E_1^2F_4}{E_1^2E_2^2}\Bigr)\cr
&=\text{Tr}_{q/2}\Bigl(\frac{D_1^2}{E_1^2}+\frac{D_1}{E_1}+\frac{F_3^2}{E_1^2E_2^2}+\frac{F_3}{E_1E_2}+\frac{F_4}{E_2^2}\Bigr)\cr
&=\text{Tr}_{q/2}\Bigl(\frac{F_4}{E_2^2}\Bigr).
\end{align}
By \eqref{4.7} and \eqref{4.10}, it is easy to check that 
\begin{equation}\label{4.17}
\frac{F_4}{E_2^2}=1+\frac 1{1+a+b}.
\end{equation}
Thus if $b=1$, we have
\[
0=\text{Tr}_{q/2}\Bigl(1+\frac 1a\Bigr),
\]
and if $a^2=b(b+1)$, we have
\[
0=\text{Tr}_{q/2}\Bigl(1+\frac 1{1+a+b}\Bigr)=\text{Tr}_{q/2}\Bigl(1+\frac 1{1+a^2+b^2}\Bigr)=\text{Tr}_{q/2}\Bigl(1+\frac 1{1+b}\Bigr)=\text{Tr}_{q/2}\Bigl(\frac b{1+b}\Bigr).
\]

\subsection{Sufficiency}\

We use the notation of Subsection 4.1. By Proposition~\ref{P2.1}, it suffices to prove the following claims.

\medskip
{\bf Claim 1.} $bX^3+X+a$ has no root in $\mu_{q+1}$.

\medskip
{\bf Claim 2.} For each $y\in\f_q$, there is a unique $x\in\f_q$ such that
\[
g\Bigl(\frac{x+z+1}{x+z}\Bigr)=\frac{y+z+1}{y+z}.
\]
By the computations in Section~3, Claim 2 is implied by the following two claims.

\medskip
{\bf Claim 2.1.} If $y\in\f_q$ is a root of $C_1C_2+C_0C_3$, it is also a root of $C_2^2+C_1C_3$.

\medskip
{\bf Claim 2.2.} There exists $D\in\f_q[Y]$ such that \eqref{3.16} holds.

\medskip
{\it Proof of Claim 1.} Assume to the contrary that there exists $x\in\mu_{q+1}$ such that 
\begin{equation}\label{4.18}
bx^3+x+a=0.
\end{equation}
Then
\begin{equation}\label{4.19}
ax^3+x^2+b=x^3(bx^3+x+a)^q=0.
\end{equation}
Eliminating the $x^3$ terms using \eqref{4.18} and \eqref{4.19}, we have $bx^2+ax+a^2+b^2=0$, i.e., 
\begin{equation}\label{4.20}
\Bigl(\frac{bx}a\Bigr)^2+\frac{bx}a+\frac{b(a^2+b^2)}{a^2}=0.
\end{equation}
On the other hand, it follows from (i) and (ii) of Theorem~\ref{T1.1} that 
\[
\text{Tr}_{q/2}\Bigl(\frac{b(a^2+b^2)}{a^2} \Bigr)=0.
\]
Thus by \eqref{4.20}, $bx/a\in\f_q$, i.e., $x\in\f_q$, whence $x=1$, which is impossible since $1+a+b\ne 0$.

\medskip
{\it Proof of Claim 2.1.} In fact, we have
\begin{align*}
C_1C_2+C_0C_3\,&=ab+(1+a^2+b^2)(Y^2+Y+k),\cr
C_2^2+C_1C_3\,&=1+b+ab+(1+a^2+b^2)(Y^2+Y+k).
\end{align*}
If $b=1$, we have $C_1C_2+C_0C_3=C_2^2+C_1C_3$. If $b\ne 1$, $\text{Tr}_{q/2}(b/(b+1))=0$ and $a^2=b(b+1)$, we have
\begin{align*}
\text{Tr}_{q/2}\Bigl(\frac{ab}{1+a^2+b^2}\Bigr)\,&=\text{Tr}_{q/2}\Bigl(\frac{ab}{1+b}\Bigr)=\text{Tr}_{q/2}\Bigl(a+\frac a{1+b}\Bigr)\cr
&=\text{Tr}_{q/2}\Bigl(\frac{a^2}{(1+b)^2}\Bigr)=\text{Tr}_{q/2}\Bigl(\frac b{1+b}\Bigr)=0.
\end{align*}
Hence $C_1C_2+C_0C_3$ has no root in $\f_q$.

\medskip
{\it Proof of Claim 2.2.} By \eqref{4.7} and \eqref{4.8}, $E_1E_2\ne 0$. By \eqref{4.17} and (i), (ii) of Theorem~\ref{T1.1}, we have 
\[
\text{Tr}_{q/2}\Bigl(\frac{F_4}{E_2^2}\Bigr)=0.
\]
It follows that there exists $D_1\in\f_q$ satisfying \eqref{3.34}. Let 
\[
D_2=\frac{E_2D_1+F_3}{E_1},\quad D_0=\frac{E_0D_1+F_1}{E_1}.
\]
Then \eqref{3.29} -- \eqref{3.33} are satisfied. Thus \eqref{3.16} holds with $D=D_2Y^2+D_1Y+D_0$. 

This completes the proof of Theorem~\ref{T1.1}.

\medskip
\noindent
{\bf Remark.} If there is an easy way to show that $a\in\f_q^*$, the proof of Theorem~\ref{T1.1} would be simplified significantly.

\section*{Appendix}

In \eqref{3.20},
\begin{align}\label{s3-F4}\tag{A1}
F_4=\,&a_1+b+b^2+a_1 b^2+b^3+b^4+a_1 b^4+b^5+b^6+a_1 b^6+b^7+b^8\\ 
&+(a_1^3 +a_1^2 b+a_1^3 b^4 +a_1^2 b^5)k+(a_1^4+a_1^5+a_1^4 b+a_1^4 b^2+a_1^5 b^2+a_1^4 b^3)k^2\cr
&+(a_1^7 +a_1^6 b)k^3+a_1^8 k^4,\nonumber
\end{align}
\begin{align}\label{s3-F3}\tag{A2}
F_3=\,& 1+a_1^2+a_1^3+a_1^4+b+a_1 b+a_1^4 b+b^2+a_1 b^2+a_1^3 b^2+b^3+b^4+a_1^2b^4\\
&+b^5+a_1 b^5+b^6+a_1 b^6+b^7+(a_1^2 +a_1^4 +a_1^5 +a_1^2 b +a_1^4 b^2 +a_1^2 b^4+a_1^2b^5)k\cr
&+(a_1^4 +a_1^4 b +a_1^5 b +a_1^4 b^2 +a_1^5 b^2 +a_1^4 b^3)k^2+(a_1^6 +a_1^6 b)k^3,
\nonumber
\end{align}
\begin{align}\label{s3-F2}\tag{A3}
F_2=\,& b+a_1 b+a_1^3 b+a_1^4 b+b^2+a_1 b^2+a_1^3 b^2+a_1^4 b^2+b^3+a_1^2 b^3+a_1^3b^3+b^4\\
&+a_1^2 b^4+a_1^3 b^4+b^5+a_1 b^5+a_1^2 b^5+b^6+a_1 b^6+a_1^2 b^6+b^7+b^8\cr
&+(a_1^3 +a_1^4 +a_1^3b +a_1^3 b^2 +a_1^5 b^2 +a_1^3 b^3 +a_1^4 b^3)k\cr
&+(a_1^4 +a_1^4 b +a_1^4 b^2 +a_1^6 b^2+a_1^4 b^3)k^2+a_1^7 k^3+a_1^8 k^4,
\nonumber
\end{align}
\begin{align}\label{s3-F1}\tag{A4}
F_1=\,& a_1+a_1^2+a_1^3+a_1^4+b+a_1 b+a_1^3 b+a_1^4 b+a_1 b^2+a_1^2
b^2+b^3\\
&+a_1^2 b^3+a_1^3 b^3+a_1 b^4+a_1^3 b^4+b^5+a_1 b^5+a_1^2 b^5+a_1 b^6+b^7\cr
&+(1+a_1^3 +a_1^4 +a_1^5 +b +a_1b +a_1^3 b +a_1^4 b +b^2 +a_1 b^2 +a_1^4 b^2\cr
& +b^3 +a_1^3 b^3 +a_1^4 b^3 +b^4 +a_1^3 b^4 +b^5
+a_1 b^5 +b^6 +a_1 b^6 +b^7)k\cr
&+(a_1^2 +a_1^4 +a_1^2 b +a_1^4 b +a_1^4 b^2 +a_1^4
b^3 +a_1^2 b^4 +a_1^2 b^5)k^2\cr
&+(a_1^4 +a_1^6 +a_1^4 b +a_1^5 b +a_1^4 b^2 +a_1^5
b^2 +a_1^4 b^3)k^3+(a_1^6 +a_1^6 b)k^4,
\nonumber
\end{align}
\begin{align}\label{s3-F0}\tag{A5}
F_0=\,& b+a_1^4 b+b^2+a_1^4 b^2+b^5+b^6+(a_1 +a_1^2 +a_1^3 +a_1^4 +b +a_1b +a_1^3 b +a_1^4 b\\
& +a_1^2 b^2 +a_1^5 b^2 +b^3 +a_1^2 b^3 +a_1^3 b^3 +a_1 b^4 +a_1^3 b^4 +b^5 +a_1b^5 +a_1^2 b^5 +b^7)k\cr
&+(a_1 +a_1^2 +a_1^4 +a_1^5 +b +a_1^3 b +a_1^4 b +a_1^5b +b^2 +a_1 b^2 +a_1^3 b^2 +a_1^4 b^2 +a_1^6 b^2\cr
& +b^3 +a_1^3 b^3 +a_1^4 b^3 +b^4 +a_1
b^4 +a_1^2 b^4 +a_1^3 b^4 +b^5 +b^6 +a_1 b^6 +b^7 +b^8)k^2\cr
&+(a_1^3 +a_1^5 +a_1^2 b +a_1^4 b +a_1^5 b^2 +a_1^4 b^3 +a_1^3 b^4 +a_1^2 b^5)k^3\cr
&+(a_1^4 +a_1^5 +a_1^6+a_1^4 b +a_1^4 b^2 +a_1^5 b^2 +a_1^4 b^3 )k^4+(a_1^7 +a_1^6 b )k^5+a_1^8 k^6.
\nonumber
\end{align}
In \eqref{3.39} and \eqref{3.40},
\begin{align}\label{s3-h1}\tag{A6}
h_1=\,& 1+a_1^4+a_1^2 b+a_1^4 b^2+a_1^2 b^5+b^8+(a_1^2 +a_1^6 +a_1^2 b^2 +a_1^2b^4 +a_1^2 b^6 )k\\
&+(a_1^4 +a_1^6 b +a_1^4 b^4 )k^2+(a_1^6 +a_1^6 b^2 )k^3,
\nonumber
\end{align}
\begin{align}\label{s3-h2}\tag{A7}
h_2=\,& a_1^2+a_1^6+b+a_1^4 b+a_1^2 b^2+a_1^6 b^2+b^3+a_1^4 b^3+a_1^2 b^4+a_1^4b^5+a_1^2 b^6+a_1^4 b^7+b^9\\
&+b^{11}+(a_1^4 +a_1^8 +a_1^4 b^4 )k+(a_1^2 +a_1^6 +a_1^4 b +a_1^8 b +a_1^6 b^4+a_1^4 b^5 +a_1^2 b^8 )k^2\cr
&+(a_1^4 +a_1^8 +a_1^4 b^2 +a_1^8 b^2 +a_1^4 b^4 +a_1^4 b^6 )k^3+(a_1^6
+a_1^8 b^3 +a_1^6 b^4 )k^4+(a_1^8 +a_1^8 b^2 )k^5.
\nonumber
\end{align}
In \eqref{3.44},
\begin{align}\label{s3-d1}\tag{A8}
d_1=\,& a_1^2+a_1^6+a_1^{10}+a_1^{14}+b+a_1^4 b+a_1^8 b+a_1^{12} b+a_1^4
b^3+a_1^{12} b^3+a_1^2 b^4+a_1^{10} b^4+b^5+a_1^8 b^5\\
&+a_1^6 b^6+a_1^8 b^7+a_1^6 b^8+a_1^8 b^{11}+a_1^6 b^{14}+a_1^2 b^{16}+b^{17}+a_1^4b^{17}+a_1^4 b^{19}+a_1^2 b^{20}+b^{21}\cr
&+(a_1^2 b +a_1^6 b +a_1^{10} b +a_1^{14} b +a_1^2 b^3 +a_1^{10} b^3 +a_1^{10}b^7 +a_1^{10} b^9 +a_1^2 b^{17} +a_1^6 b^{17} +a_1^2 b^{19} )k\cr
&+(a_1^2 +a_1^6 +a_1^{10} +a_1^{14}+a_1^2 b^2 +a_1^6 b^2 +a_1^{10} b^2 +a_1^{14} b^2 +a_1^8 b^3 +a_1^{12} b^3 +a_1^6 b^4 +a_1^{10}b^4 \cr
&+a_1^{12} b^5 +a_1^6 b^6 +a_1^6 b^8 +a_1^6 b^{10} +a_1^{10} b^{10} +a_1^8 b^{11} +a_1^6 b^{12} +a_1^6b^{14} +a_1^2 b^{16} +a_1^2 b^{18} )k^2\cr
&+(a_1^{10} b^3 +a_1^{14} b^3 +a_1^{10} b^5 +a_1^{10} b^7 +a_1^{10} b^9 )k^3\cr
&+(a_1^6+a_1^{14} +a_1^6 b^2 +a_1^{14} b^2 +a_1^6 b^4 +a_1^6 b^6 +a_1^6 b^8 +a_1^6 b^{10} +a_1^6
b^{12} +a_1^6 b^{14} )k^4,
\nonumber
\end{align}
\begin{align}\label{s3-d2}\tag{A9}
d_2=\,& 1+a_1^4+a_1^8+a_1^{12}+a_1^2 b+a_1^{10} b+b^2+a_1^8 b^2+a_1^6 b^3+a_1^{10}b^3\\
&+a_1^8 b^6+a_1^8 b^8+a_1^6 b^{11}+b^{16}+a_1^4 b^{16}+a_1^2 b^{17}+b^{18}\cr
&+(a_1^4 b +a_1^{12} b +a_1^4 b^3 +a_1^8 b^3
+a_1^4 b^5 +a_1^8 b^5 +a_1^4 b^7\cr 
&+a_1^8 b^7 +a_1^4 b^9 +a_1^8 b^9 +a_1^4 b^{11} +a_1^4 b^{13} +a_1^4
b^{15} )k\cr
&+(a_1^4 +a_1^{12} +a_1^4 b^2 +a_1^{12} b^2 +a_1^4 b^4 +a_1^4 b^6 +a_1^4 b^8 +a_1^4
b^{10} +a_1^4 b^{12} +a_1^4 b^{14} )k^2.
\nonumber
\end{align}



\end{document}